\documentclass[11pt]{article}

\usepackage[a4paper,margin=28mm]{geometry}

\usepackage[T1]{fontenc}
\usepackage[utf8]{inputenc}
\usepackage{lmodern}
\usepackage{amsmath,amssymb,amsthm,mathtools}
\usepackage{aliascnt}
\usepackage{mathrsfs}

\usepackage{enumitem}

\usepackage[colorlinks=true,linkcolor=blue,citecolor=blue,urlcolor=blue]{hyperref}
\usepackage[nameinlink,capitalize]{cleveref}

\theoremstyle{plain}
\newtheorem{theorem}{Theorem}[section]
\newaliascnt{lemma}{theorem}
\newtheorem{lemma}[lemma]{Lemma}
\aliascntresetthe{lemma}
\newaliascnt{proposition}{theorem}
\newtheorem{proposition}[proposition]{Proposition}
\aliascntresetthe{proposition}
\newaliascnt{corollary}{theorem}
\newtheorem{corollary}[corollary]{Corollary}
\aliascntresetthe{corollary}

\theoremstyle{definition}
\newaliascnt{definition}{theorem}
\newtheorem{definition}[definition]{Definition}
\aliascntresetthe{definition}
\newaliascnt{assumption}{theorem}

\aliascntresetthe{assumption}
\newaliascnt{example}{theorem}
\newtheorem{example}[example]{Example}
\aliascntresetthe{example}

\theoremstyle{remark}
\newaliascnt{remark}{theorem}
\newtheorem{remark}[remark]{Remark}
\aliascntresetthe{remark}

\newcommand{\R}{\mathbb{R}}

\newcommand{\Hr}{\mathcal{H}_{\mathrm r}}
\newcommand{\Hu}{\mathcal{H}_{\mathrm u}}
\newcommand{\Hfull}{\mathcal{H}}
\newcommand{\Lop}{\mathcal{L}}

\newcommand{\norm}[1]{\left\lVert #1 \right\rVert}
\newcommand{\inner}[2]{\left( #1,#2 \right)}
\newcommand{\diff}{\mathop{}\!\mathrm{d}}
\newcommand{\Id}{\mathrm{I}}
\newcommand{\e}{\mathrm{e}}

\title{Stable Reduction of Unresolved Dynamics: An Operator-Theoretic Framework}

\author{
  Hiroki Ishizaka\\
  Team FEM, Matsuyama, Japan\\
  E-mail: \texttt{h.ishizaka005@gmail.com}
}

\date{}

\begin{document}

\maketitle

\begin{abstract}
Dimension reduction removes variables, but a reliable reduction should not erase their dynamical influence. We study this principle for linear block evolution systems on resolved and unresolved Hilbert spaces. Exact elimination of the unresolved component yields a Volterra equation with a memory kernel and an effective forcing term carrying hidden initial data and forcing; in the Laplace domain the same operation is a dynamic Schur complement. We introduce finite-horizon trajectory-stable reduction and derive perturbation estimates for operator-valued kernels, then lift them to perturbations of the hidden propagator, couplings, hidden initial state, and hidden forcing. For passive skew-adjoint couplings, the reduced memory operator has positive type and an exact storage identity, and norm-convergent passive realisations preserve both trajectories and the storage--dissipation functional. Compatible finite spectral truncations provide structure-preserving internal-variable reductions. Two elementary examples separate positive type from instantaneous $L^2$-coercivity and show that small hidden amplitude need not imply small long-time influence. The framework isolates the controlled approximation of unresolved influence, rather than elimination alone, as the central reduction problem.
\end{abstract}



\section{Introduction}

Many mathematical models contain more variables, scales, or degrees of freedom than one wishes to retain in an effective description. Model reduction therefore seeks a smaller system which reproduces selected aspects of the original dynamics at lower computational or conceptual cost; see, for example, \cite{BennerGugercinWillcox2015,GivonKupfermanStuart2004,PavliotisStuart2008}. The central difficulty, however, is not merely to remove variables. Variables which are not resolved may still feed back into those which are retained.

This observation is classical in nonequilibrium statistical mechanics. Projection-operator approaches going back to Zwanzig \cite{Zwanzig1961} and Mori \cite{Mori1965} show that eliminating degrees of freedom naturally produces memory and fluctuating or effective forcing terms. The same principle is central to the Mori--Zwanzig representation and to optimal prediction for underresolved systems \cite{ChorinKastKupferman1999,ChorinHaldKupferman2000,ChorinHaldKupferman2002}. Rigorous questions concerning the orthogonal dynamics were studied, for instance, by Givon, Hald, and Kupferman \cite{GivonHaldKupferman2005}. From a different but complementary viewpoint, abstract Volterra equations provide a natural analytic framework for evolution equations with memory; see \cite{GripenbergLondenStaffans1990,Pruss1993}.

The point of the present note is deliberately simple:
\begin{quote}
\emph{Dimension reduction removes variables. Stable reduction removes variables without erasing their dynamical influence.}
\end{quote}
The statement is not meant as a slogan replacing existing reduction theories. Rather, it identifies the property that will be used here to organise them. In particular, the issue is not whether an unresolved component is small in norm at one instant, but whether its influence on the resolved dynamics is small in the sense relevant to the evolution problem.

We begin with a linear block evolution system
\begin{align*}
  \Hfull=\Hr\oplus\Hu,
\end{align*}
where $\Hr$ contains the resolved variables and $\Hu$ the unresolved variables. Exact elimination of the unresolved component produces a Volterra equation on $\Hr$. Two objects remain after the unresolved variable itself has disappeared: a memory kernel and an effective forcing term. The latter contains, in particular, the unresolved initial state. This elementary identity already captures the mechanism that motivates the terminology \emph{stable reduction}.

The second purpose of the note is to separate this structural observation from the approximation problem. Once the exact reduced equation is written in memory form, one may approximate its unresolved influence or replace the memory representation by internal variables. Such operations are useful only if the resolved dynamics remains stable with respect to the reduction. We first prove a finite-horizon perturbation estimate in the natural $L^1$-topology for operator-valued kernels and then lift this estimate back to perturbations of the hidden propagator, couplings, hidden initial state, and hidden forcing. In the passive class we impose an additional storage--dissipation requirement, thereby separating trajectory stability from structure preservation.

The present framework also connects with the author's recent treatment of coercive evolution equations with measure-valued delays \cite{Ishizaka2026Memory}, where stability is studied for prescribed delay measures. The current note moves in the opposite direction: it asks how temporal nonlocality can arise from eliminating unresolved dynamics.

The main points of the note are as follows.
\begin{enumerate}[label=\textup{(\roman*)},leftmargin=2.6em]
  \item A block evolution system admits an exact elimination identity in which the unresolved dynamics survives as a memory kernel and an unresolved-state forcing term.
  \item The same elimination is identified in the Laplace domain as a dynamic Schur complement. Finite-dimensional hidden dynamics yield rational memory transfer functions; non-diagonalisable and oscillatory hidden modes are distinguished from pure sums of exponentials.
  \item A finite-horizon notion of trajectory-stable reduction is introduced for memory-bearing reduced equations, together with a kernel-stability estimate for perturbations integrable in operator norm.
  \item Stability is lifted to the hidden realisation: perturbations of the unresolved propagator, couplings, initial state, and hidden forcing are propagated quantitatively to the resolved trajectory.
  \item For a passive skew-adjoint coupling, the left-hand memory operator has positive type and carries an exact internal-storage identity. Norm-convergent passive realisations preserve this storage--dissipation structure, while positive type alone is shown not to imply instantaneous $L^2$-coercivity.
  \item Finite spectral truncations provide a concrete class of structure-preserving internal-variable reductions.
  \item A simple example shows that small unresolved amplitude does not, by itself, imply a uniformly small dynamical influence on long time intervals.
\end{enumerate}

The framework is intentionally restricted. We do not claim that every useful reduced model should be non-Markovian, nor that the Mori--Zwanzig formalism is the only route to memory. The aim is to isolate a minimal operator-theoretic mechanism and the stability questions attached to it.

\noindent\emph{Organisation of the paper.}
\Cref{sec:block} derives the exact memory equation obtained by eliminating the unresolved component. \Cref{sec:structure} records the corresponding dynamic Schur complement and isolates a passive class with a memory-storage identity. \Cref{sec:stable} introduces finite-horizon trajectory-stable reduction and proves the basic stability estimates. \Cref{sec:realisation} lifts these estimates to perturbations of the hidden realisation and gives structure-preserving passive reductions, including a finite-mode truncation. \Cref{sec:examples} gives elementary examples, while \cref{sec:meaning,sec:scope} state the rôle and scope of the note.

\section{Resolved and unresolved dynamics}\label{sec:block}
Throughout the note, $\Hr$ and $\Hu$ are separable real Hilbert spaces. Whenever Laplace transforms or complex spectra are used, we pass to the canonical complexifications of the spaces and operators, without changing notation. This convention keeps the energy identities below in real inner-product form while allowing the usual complex resolvent and spectral calculations.

We set
\begin{align*}
  \Hfull:=\Hr\oplus\Hu,
  \quad
  P_{\mathrm r}:\Hfull\to\Hr,
  \quad
  P_{\mathrm r}(x,y):=x.
\end{align*}
We write $x(t)\in\Hr$ for the resolved state and $y(t)\in\Hu$ for the unresolved state, so that $P_{\mathrm r}(x(t),y(t))=x(t)$. We consider the block evolution system
\begin{subequations}\label{eq:block-system}
\begin{align}
  \frac{\mathrm{d}}{\mathrm{d}t}x(t) &= A x(t)+B y(t)+f(t), \label{eq:block-x}\\
  \frac{\mathrm{d}}{\mathrm{d}t}y(t) &= C x(t)+D y(t)+g(t), \label{eq:block-y}
\end{align}
\end{subequations}
with initial values
\begin{align}\label{eq:block-initial}
  x(0)=x_0,\quad y(0)=y_0.
\end{align}
Here,
\begin{align*}
  A:D(A)\subset\Hr\to\Hr,
  \quad
  D:D(D)\subset\Hu\to\Hu
\end{align*}
are possibly unbounded operators, while
\begin{align*}
  B\in\Lop(\Hu,\Hr),
  \quad
  C\in\Lop(\Hr,\Hu)
\end{align*}
are bounded coupling operators. For the exact classical derivation below, $f$ and $g$ are assumed to have whatever regularity is needed for the stated classical solution; the stability theory later requires only $L^1$-in-time forcing. We assume that $A$ and $D$ generate strongly continuous semigroups $S_A(t)$ and $S_D(t)$, respectively. By the bounded perturbation theorem, the corresponding block operator
\begin{align*}
  \mathcal A
  :=
  \begin{pmatrix}
    A & B\\
    C & D
  \end{pmatrix},
  \quad
  D(\mathcal A)=D(A)\oplus D(D),
\end{align*}
generates a strongly continuous semigroup on $\Hfull$; see, for example, \cite[Chap.~3]{Pazy1983}. We shall use only the simpler elimination identity below.

For an operator-valued kernel $K:[0,\infty)\to\Lop(\Hr)$, write
\begin{align*}
  (K*x)(t):=\int_0^t K(t-s)x(s)\,\diff s
\end{align*}
whenever the integral is well defined.

\begin{proposition}[Exact elimination of the unresolved component]\label{prop:exact-elimination}
Assume that $(x,y)$ is a classical solution of \cref{eq:block-system,eq:block-initial}. Then, the resolved component satisfies
\begin{align}\label{eq:exact-reduced}
  \frac{\mathrm{d}}{\mathrm{d}t}x(t)
  =A x(t)+(K*x)(t)+F_{\mathrm u}(t)+f(t),
  \quad x(0)=x_0,
\end{align}
where
\begin{align}
  K(t)&:=B S_D(t)C, \label{eq:kernel}\\
  F_{\mathrm u}(t)
  &:=B S_D(t)y_0
    +B\int_0^t S_D(t-s)g(s)\,\diff s. \label{eq:unresolved-forcing}
\end{align}
Conversely, if $x$ solves \cref{eq:exact-reduced} and
\begin{align}\label{eq:reconstruct-y}
  y(t)
  :=S_D(t)y_0
    +\int_0^t S_D(t-s)\{Cx(s)+g(s)\}\,\diff s,
\end{align}
with sufficient regularity for \cref{eq:block-system} to hold classically, then $(x,y)$ solves the original block system.
\end{proposition}

\begin{proof}
Variation of constants in the unresolved equation \cref{eq:block-y} gives
\begin{align*}
  y(t)
  =S_D(t)y_0
   +\int_0^t S_D(t-s)Cx(s)\,\diff s
   +\int_0^t S_D(t-s)g(s)\,\diff s.
\end{align*}
Substitution into \cref{eq:block-x} yields
\begin{align*}
\begin{split}
  \frac{\mathrm{d}}{\mathrm{d}t}x(t)
  = A x(t)
      +B S_D(t)y_0
      +\int_0^t B S_D(t-s)C x(s)\,\diff s 
     +B\int_0^t S_D(t-s)g(s)\,\diff s
      +f(t),
\end{split}
\end{align*}
which is exactly \cref{eq:exact-reduced}--\cref{eq:unresolved-forcing}. The converse follows by reversing the same calculation.
\end{proof}

\begin{remark}[Memory is generated by elimination]\label{rem:memory-generated}
The kernel $K(t)=BS_D(t)C$ is not an additional constitutive assumption. It is the exact feedback path
\begin{align*}
  \Hr \xrightarrow{\ C\ } \Hu
  \xrightarrow{\ S_D(t)\ } \Hu
  \xrightarrow{\ B\ } \Hr
\end{align*}
through the unresolved dynamics. Thus, even when the original block system is local in time, its exact equation on the resolved space is generally nonlocal in time. This is the linear block-system analogue of the memory mechanism underlying projection-operator and Mori--Zwanzig formulations \cite{Mori1965,Zwanzig1961,ChorinHaldKupferman2000,ChorinHaldKupferman2002}.
\end{remark}

\begin{remark}[The unresolved initial state does not disappear]\label{rem:initial-unresolved}
The term $BS_D(t)y_0$ in \cref{eq:unresolved-forcing} records information carried by the unresolved initial state. Therefore, removing $y$ from the state vector does not justify removing $y_0$ from the dynamics. If $y_0$ is unknown, it must be estimated, modelled statistically, bounded, or otherwise represented. Simply setting it to zero is an additional modelling assumption. Likewise, writing a Volterra convolution with lower limit zero and no initial-memory term selects a particular initialisation of the memory. It should not be identified with zero physical prehistory unless a realisation of the memory state has been specified.
\end{remark}

\begin{remark}[Mild formulation]\label{rem:mild-formulation}
The classical regularity in \cref{prop:exact-elimination} is used only to keep the derivation transparent. Under the standard semigroup assumptions, the same identity can be written in mild form. This is the natural setting for unbounded differential operators and is consistent with the usual semigroup treatment of abstract evolution equations \cite{Pazy1983} and with the Volterra framework in \cite{GripenbergLondenStaffans1990,Pruss1993}.
\end{remark}

\section{Structural information carried by exact elimination}\label{sec:structure}
The time-domain identity in \cref{prop:exact-elimination} has two useful complementary interpretations. The first is a frequency-domain Schur complement. The second is an energy interpretation available under additional symmetry and passivity assumptions. Neither interpretation should be imposed on a general reduction without checking its hypotheses.

\subsection{Dynamic Schur complement}

For a function of exponential order, write its Laplace transform as
\begin{align*}
  \widehat v(\zeta):=\int_0^\infty \e^{-\zeta t}v(t)\,\diff t
\end{align*}
for $\operatorname{Re}\zeta$ sufficiently large.

\begin{proposition}[Laplace-domain elimination]\label{prop:dynamic-schur}
Assume that $(x,y)$ is a classical solution of \cref{eq:block-system} of exponential order and that $Ax,Dy,f,g$ are also of exponential order, so that the Laplace transform may be applied term by term in the complexified spaces. Then,
\begin{align}\label{eq:dynamic-schur}
  \left[
    \zeta\Id-A-B(\zeta\Id-D)^{-1}C
  \right]\widehat x(\zeta)
  =x_0+\widehat f(\zeta)
   +B(\zeta\Id-D)^{-1}\{y_0+\widehat g(\zeta)\}.
\end{align}
Furthermore,
\begin{align}\label{eq:laplace-kernel}
  \widehat K(\zeta)
  =B(\zeta\Id-D)^{-1}C.
\end{align}
Thus, the memory transfer operator is the frequency-dependent contribution produced by eliminating the unresolved block.
\end{proposition}

\begin{proof}
Taking Laplace transforms in \cref{eq:block-y} gives
\begin{align*}
  (\zeta\Id-D)\widehat y(\zeta)
  =y_0+C\widehat x(\zeta)+\widehat g(\zeta).
\end{align*}
For $\operatorname{Re}\zeta$ in the resolvent half-plane of $D$, solve for $\widehat y$ and substitute the result into the transformed resolved equation. This yields \cref{eq:dynamic-schur}. The semigroup resolvent formula
\begin{align*}
  \int_0^\infty \e^{-\zeta t}S_D(t)\,\diff t
  =(\zeta\Id-D)^{-1}
\end{align*}
then gives \cref{eq:laplace-kernel}; see, for example, \cite{Pazy1983,Staffans2005}.
\end{proof}

\begin{remark}[Why ``dynamic Schur complement'' is useful]\label{rem:dynamic-schur}
For a static block matrix, elimination produces an ordinary Schur complement. In \cref{eq:dynamic-schur}, the unresolved block contributes the resolvent $(\zeta\Id-D)^{-1}$, hence a frequency-dependent Schur complement. Its inverse Laplace transform is precisely the memory kernel. This makes explicit that time nonlocality and frequency dependence are two representations of the same eliminated dynamics.
\end{remark}

\subsection{Finite-dimensional hidden modes}

\begin{proposition}[Finite-dimensional realisation]\label{prop:finite-modes}
Suppose that $\Hu$ is finite-dimensional and that the complexification of $D$ is diagonalisable. After complexification, we write
\begin{align*}
  D=\sum_{j=1}^{m}\lambda_j P_j,
\end{align*}
where $P_j$ are the spectral projections. Then,
\begin{align}\label{eq:finite-mode-kernel}
  K(t)=\sum_{j=1}^{m}\e^{\lambda_j t}BP_jC.
\end{align}
In particular, exponentially decaying hidden modes produce a finite sum of decaying exponential memory terms.
\end{proposition}

\begin{proof}
The spectral decomposition gives
$
  S_D(t)=\e^{tD}=\sum_j\e^{\lambda_jt}P_j
$.
Substitution into \cref{eq:kernel} gives \cref{eq:finite-mode-kernel}.
\end{proof}

\begin{remark}[Finite-dimensional does not mean ``pure Prony'' in general]\label{rem:jordan-modes}
The diagonalisability hypothesis in \cref{prop:finite-modes} is essential for the pure exponential form. A Jordan block produces terms of the form $t^q\e^{\lambda t}$, while non-real eigenvalues produce oscillatory real combinations. Equivalently, the transfer operator in \cref{eq:laplace-kernel} is rational in finite dimensions, but its inverse Laplace transform need not be a positive sum of real decaying exponentials. Thus a Prony representation is a structural subclass, not a consequence of finite-dimensional elimination alone; compare the realisation viewpoint in \cite{Staffans2005}.
\end{remark}

\subsection{Passive coupling, memory storage, and positive type}

The general right-hand memory kernel $BS_D(t)C$ in \cref{eq:exact-reduced} need not be positive or completely monotone. For the passive sign convention, the off-diagonal part of the block generator is skew-adjoint: one has $B=-C^*$ and $D=-L$ with $L=L^*\ge0$. The exact right-hand kernel is then $-C^*\e^{-tL}C$. After moving this feedback term to the left-hand side, the associated memory operator is the positive kernel defined below. We say that a left-hand kernel $K$ is \emph{of positive type on $[0,\mathfrak T]$} if
\begin{align*}
  \int_0^{\mathfrak T}\inner{(K*x)(t)}{x(t)}\,\diff t\ge0
  \quad
  \text{for any }x\in L^2(0,\mathfrak T;\Hr).
\end{align*}

\begin{theorem}[Passive elimination and memory storage]\label{thm:passive-memory}
Let $L\in\Lop(\Hu)$ be self-adjoint and non-negative, and let
$C\in\Lop(\Hr,\Hu)$. We define
\begin{align}\label{eq:passive-kernel}
  K_L(t):=C^*\e^{-tL}C,
  \quad t\ge0.
\end{align}
For $x\in L^2(0,\mathfrak T;\Hr)$, we set
\begin{align}\label{eq:memory-state-q}
  q(t):=\int_0^t\e^{-(t-s)L}Cx(s)\,\diff s.
\end{align}
Then, $q\in H^1(0,\mathfrak T;\Hu)$,
\begin{align*}
  \frac{\mathrm{d}}{\mathrm{d}t}q +Lq=Cx,
  \quad q(0)=0,
\end{align*}
and
\begin{align}\label{eq:positive-type-identity}
\begin{split}
  \int_0^{\mathfrak T}
  \inner{(K_L*x)(t)}{x(t)}\,\diff t
  = \frac12\norm{q(\mathfrak T)}_{\Hu}^2
  +\int_0^{\mathfrak T} \norm{L^{1/2}q(t)}_{\Hu}^2\,\diff t
  \ge0.
\end{split}
\end{align}
Therefore, $K_L$ is of positive type on any finite time interval.
\end{theorem}

\begin{proof}
Because $L$ is bounded, differentiation of \cref{eq:memory-state-q} in $L^2$ gives
\begin{align*}
  \frac{\mathrm{d}}{\mathrm{d}t}q+Lq=Cx.
\end{align*}
Furthermore,
\begin{align*}
  (K_L*x)(t)=C^*q(t).
\end{align*}
Therefore,
\begin{align*}
  \int_0^{\mathfrak T}\inner{(K_L*x)(t)}{x(t)}\,\diff t
  &=\int_0^{\mathfrak T}\inner{q(t)}{Cx(t)}\,\diff t
  =\int_0^{\mathfrak T}\inner{q(t)}{ \frac{\mathrm{d}}{\mathrm{d}t}q(t) +Lq(t)} \,\diff t,
  \end{align*}
which is exactly \cref{eq:positive-type-identity}.
\end{proof}

\begin{corollary}[Energy identity after passive elimination]\label{cor:passive-energy}
Assume, in addition, that $A_0\in\Lop(\Hr)$ is self-adjoint and non-negative. We consider
\begin{subequations}\label{eq:passive-extended-system}
\begin{align}
  \frac{\mathrm{d}}{\mathrm{d}t}x +A_0x+C^*y&=f,\label{eq:passive-x}\\
  \frac{\mathrm{d}}{\mathrm{d}t}y+Ly-Cx&=0,\label{eq:passive-y}
\end{align}
\end{subequations}
with $y(0)=0$. Eliminating $y$ gives
\begin{align}\label{eq:passive-reduced}
  \frac{\mathrm{d}}{\mathrm{d}t}x+A_0x+(K_L*x)=f.
\end{align}
For any classical solution,
\begin{align}\label{eq:passive-energy}
  \frac12\frac{\mathrm{d}}{\mathrm{d}t}
  \left(
    \norm{x(t)}_{\Hr}^2+\norm{y(t)}_{\Hu}^2
  \right)
  +\inner{A_0x(t)}{x(t)}
  +\norm{L^{1/2}y(t)}_{\Hu}^2
  =\inner{f(t)}{x(t)}.
\end{align}
Thus, the hidden state may be removed from the list of resolved variables without removing its stored energy: in the reduced formulation, storage is encoded by the memory state $q=y$. Such history-space energies are classical in the analysis of Volterra and viscoelastic equations; see, for example, \cite{Dafermos1970,Pruss1993}.
\end{corollary}

\begin{proof}
By variation of constants in \cref{eq:passive-y}, $y=q$ from \cref{eq:memory-state-q}; substitution into \cref{eq:passive-x} gives \cref{eq:passive-reduced}. Testing \cref{eq:passive-x} by $x$ and \cref{eq:passive-y} by $y$, and then adding the two equalities, cancels the coupling terms because
$
  \inner{C^*y}{x}=\inner{y}{Cx}
$.
This gives \cref{eq:passive-energy}.
\end{proof}

\begin{corollary}[Complete monotonicity of the passive memory kernel]\label{cor:complete-monotonicity}
Under the assumptions of \cref{thm:passive-memory}, for any $v\in\Hr$ the scalar function
\begin{align*}
  k_v(t):=\inner{K_L(t)v}{v}
\end{align*}
is completely monotone on $(0,\infty)$. More precisely, if $E_L$ is the spectral resolution of $L$, then
\begin{align}\label{eq:bernstein-spectral}
  k_v(t)
  =\int_{[0,\norm{L}_{\Lop(\Hu)}]}
    \e^{-\lambda t}\,\diff\nu_v(\lambda),
  \quad
  \nu_v(E):=\norm{E_L(E)Cv}_{\Hu}^2.
\end{align}
\end{corollary}

\begin{proof}
The spectral theorem gives
\begin{align*}
  \inner{\e^{-tL}Cv}{Cv}
  =\int_{[0,\norm{L}_{\Lop(\Hu)}]}\e^{-\lambda t}\,\diff\nu_v(\lambda).
\end{align*}
The measure $\nu_v$ is finite and non-negative. Bernstein's theorem then gives complete monotonicity; see \cite{SchillingSongVondracek2012}.
\end{proof}

\begin{remark}[Complete monotonicity is not automatic]\label{rem:cm-not-automatic}
\Cref{cor:complete-monotonicity} uses both the self-adjoint non-negative hidden operator and the adjoint coupling $C^*,C$. For the general kernel $K(t)=BS_D(t)C$ of \cref{eq:kernel}, these assumptions need not hold. The kernel may change sign, be non-self-adjoint, or contain damped oscillations. Thus complete monotonicity is a property to be derived from the hidden dynamics, not an axiom produced by elimination itself.
\end{remark}

\begin{proposition}[Positive type does not imply instantaneous coercivity]\label{prop:no-coercivity}
Let $\Hr\ne\{0\}$, $\gamma>0$, and
\begin{align*}
  K_\gamma(t):=\gamma\e^{-\gamma t}\Id.
\end{align*}
Then,
\begin{align*}
  \int_0^{\mathfrak T}\inner{(K_\gamma*x)(t)}{x(t)}\,\diff t\ge0
\end{align*}
for any $x\in L^2(0,\mathfrak T;\Hr)$, but there is no constant $c_{\mathfrak T}>0$ such that
\begin{align}\label{eq:no-instantaneous-coercivity}
  \int_0^{\mathfrak T}\inner{(K_\gamma*x)(t)}{x(t)}\,\diff t
  \ge c_{\mathfrak T}\norm{x}_{L^2(0,\mathfrak T;\Hr)}^2
\end{align}
for all $x$.
\end{proposition}

\begin{proof}
Positivity follows from \cref{thm:passive-memory} by taking $\Hu=\Hr$, $L=\gamma\Id$, and $C=\sqrt{\gamma}\Id$. Fix $v\in\Hr$ with $\norm{v}_{\Hr}=1$ and let $x_n(t)=\sin(nt)v$. We set
\begin{align*}
  z_n(t):=\int_0^t\e^{-\gamma(t-s)}\sin(ns)\,\diff s
  =\frac{\gamma\sin(nt)-n\cos(nt)+n\e^{-\gamma t}}
         {\gamma^2+n^2}.
\end{align*}
Then, $\norm{z_n}_{L^\infty(0,\mathfrak T)}=O(n^{-1})$. Because
\begin{align*}
  \int_0^{\mathfrak T}\inner{(K_\gamma*x_n)(t)}{x_n(t)}\,\diff t
  =\frac{\gamma}{2}|z_n(\mathfrak T)|^2
   +\gamma^2\int_0^{\mathfrak T}|z_n(t)|^2\,\diff t,
\end{align*}
the left-hand side is $O(n^{-2})$, whereas
\begin{align*}
  \norm{x_n}_{L^2(0,\mathfrak T;\Hr)}^2
  =\int_0^{\mathfrak T}\sin^2(nt)\,\diff t
  \longrightarrow \frac{\mathfrak T}{2}.
\end{align*}
Therefore, \cref{eq:no-instantaneous-coercivity} cannot hold with any $c_{\mathfrak T}>0$.
\end{proof}

\begin{remark}[First structural lesson]\label{rem:structure-level}
Trajectory stability and energy compatibility are distinct requirements. \Cref{thm:kernel-stability} below controls the resolved trajectory under perturbation of a memory representation, whereas \cref{cor:passive-energy} identifies a storage--dissipation structure inherited from a particular hidden realisation. A reduction can satisfy one property without automatically satisfying the other.
\end{remark}

\section{Finite-horizon stable reduction}\label{sec:stable}
The exact kernel $K$ in \cref{eq:kernel} need not be cheap to evaluate. Indeed, computing $S_D(t)$ may be no easier than solving the unresolved problem itself. Reduction therefore begins only after the exact influence has been identified. One may approximate $K$, replace it by a finite internal-variable system, truncate its memory, or seek a Markovian approximation. The basic requirement is that the resolved trajectory should remain controlled under such operations.

We first isolate the reduced Volterra problem. Let $A:D(A)\subset\Hr\to\Hr$ generate a strongly continuous semigroup $S_A(t)$ satisfying
\begin{align}\label{eq:semigroup-bound}
  \norm{S_A(t)}_{\Lop(\Hr)}\le M\e^{\omega t}
  \quad (t\ge0)
\end{align}
for some $M\ge1$ and $\omega\in\R$.

We shall use the following kernel class. Let $\mathscr K_{\mathfrak T}(\Hr)$ consist of operator-valued functions $K:[0, \mathfrak T]\to\Lop(\Hr)$ such that $t\mapsto K(t)v$ is strongly measurable in $\Hr$ for every $v\in\Hr$, and
\begin{align*}
  \norm{K}_{\mathscr K_{\mathfrak T}(\Hr)}
  :=\int_0^{\mathfrak T}\norm{K(t)}_{\Lop(\Hr)}\,\diff t<\infty.
\end{align*}
Because $\Hr$ is separable, the operator norm of a strongly-operator-measurable kernel is measurable: it is the supremum over a countable dense subset of the unit sphere of the measurable functions $t\mapsto\norm{K(t)v}_{\Hr}$. We identify kernels whose difference has zero $\mathscr K_{\mathfrak T}$-norm; with this a.e. identification, the displayed quantity is a genuine norm. All convolutions below are Bochner integrals in the state space. This strong-operator formulation avoids an unnecessary Bochner-measurability assumption in the generally non-separable operator space $\Lop(\Hr)$. If $S_D$ is a strongly continuous semigroup and $B,C$ are bounded, then $t\mapsto BS_D(t)C$ belongs to $\mathscr K_{\mathfrak T}(\Hr)$ on any finite interval: strong measurability follows from strong continuity and local boundedness of the semigroup gives integrability of the measurable operator norm.

\begin{lemma}[Kernel convolution calculus]\label{lem:kernel-convolution}
Let $K\in\mathscr K_{\mathfrak T}(\Hr)$ and $z\in \mathcal{C}([0,\mathfrak T];\Hr)$. For each $t\in[0,\mathfrak T]$, the map
\begin{align*}
s \longmapsto K(t-s)z(s)
\end{align*}
is strongly measurable and Bochner integrable on $(0,t)$, up to modification on a null set. Therefore,
\begin{align*}
  (K*z)(t):=\int_0^t K(t-s)z(s)\,\diff s
\end{align*}
is well defined for any $t\in[0,\mathfrak T]$. Then, $K * z \in \mathcal{C}([0,\mathfrak T];\Hr)$, and satisfies
\begin{align}\label{eq:kernel-convolution-bound}
  \norm{K*z}_{\mathcal{C}([0,\mathfrak T];\Hr)}
  \le
  \norm{K}_{\mathscr K_{\mathfrak T}(\Hr)}
  \norm{z}_{\mathcal{C}([0,\mathfrak T];\Hr)}.
\end{align}
More generally, for $K_1,K_2\in\mathscr K_{\mathfrak T}(\Hr)$ and $z_1,z_2\in \mathcal{C}([0,\mathfrak T];\Hr)$,
\begin{align}\label{eq:kernel-convolution-difference}
\begin{split}
  \norm{K_1*z_1-K_2*z_2}_{\mathcal{C}([0,\mathfrak T];\Hr)}
  \le{}&
  \norm{K_1-K_2}_{\mathscr K_{\mathfrak T}(\Hr)}
  \norm{z_1}_{\mathcal{C}([0,\mathfrak T];\Hr)}\\
  &+\norm{K_2}_{\mathscr K_{\mathfrak T}(\Hr)}
  \norm{z_1-z_2}_{\mathcal{C}([0,\mathfrak T];\Hr)}.
\end{split}
\end{align}
\end{lemma}

\begin{proof}
Because the range of a continuous function on $[0,\mathfrak T]$ is compact and hence separable, $z$ can be approximated uniformly by simple functions taking finitely many values. The assumed strong measurability of $K(\cdot)v$ for each fixed $v$, followed by this approximation, shows that $s\mapsto K(t-s)z(s)$ is strongly measurable. The bound
\begin{align*}
  \norm{K(t-s)z(s)}_{\Hr}
  \le
  \norm{K(t-s)}_{\Lop(\Hr)}\,\norm{z}_{\mathcal{C}([0,\mathfrak T];\Hr)}
\end{align*}
gives Bochner integrability. Writing equivalently
\begin{align*}
  (K*z)(t)=\int_0^t K(r)z(t-r)\,\diff r,
\end{align*}
uniform continuity of $z$, together with absolute continuity of the integral of $\norm{K(r)}_{\Lop(\Hr)}$, proves continuity in $t$. Estimate \cref{eq:kernel-convolution-bound} follows directly. Finally,
\begin{align*}
  K_1*z_1-K_2*z_2
  =(K_1-K_2)*z_1+K_2*(z_1-z_2),
\end{align*}
and a second application of \cref{eq:kernel-convolution-bound} gives \cref{eq:kernel-convolution-difference}.
\end{proof}

On a fixed interval $[0,\mathfrak T]$, we consider
\begin{align}\label{eq:volterra-reduced}
  \frac{\mathrm{d}}{\mathrm{d}t}z(t)
  =Az(t)+(K*z)(t)+F(t),
  \quad z(0)=z_0,
\end{align}
where
\begin{align*}
  K\in \mathscr K_{\mathfrak T}(\Hr),
  \quad
  F\in L^1(0,\mathfrak T;\Hr).
\end{align*}

\begin{definition}[Finite-horizon stable reduction]\label{def:stable-reduction}
Let $(K,F)$ denote the exact memory and effective forcing associated with a resolved model on $[0,\mathfrak T]$. A family of reduced models
\begin{align}\label{eq:reduction-family}
  \frac{\mathrm{d}}{\mathrm{d}t}z_n(t)
  =Az_n(t)+(K_n*z_n)(t)+F_n(t),
  \quad z_n(0)=z_{0,n},
\end{align}
is called a \emph{stable reduction on $[0,\mathfrak T]$} if the following properties hold.
\begin{enumerate}[label=\textup{(S\arabic*)},leftmargin=3.2em]
  \item Each reduced problem is well posed in the chosen trajectory space.
  \item Uniformly bounded reduction data produce uniformly bounded resolved trajectories.
  \item Consistent approximation of the memory, effective forcing, and initial data produces convergence of the resolved trajectories; in particular,
  \begin{align*}
    K_n\to K,
    \quad
    F_n\to F,
    \quad
    z_{0,n}\to z_0
  \end{align*}
  in the specified data topology implies $z_n\to z$ in the trajectory topology.
\end{enumerate}
In the present note, the data topology is
\begin{align*}
  \mathscr K_{\mathfrak T}(\Hr)
  \times L^1(0,\mathfrak T;\Hr)\times\Hr,
\end{align*}
and the trajectory topology is $\mathcal{C}([0,\mathfrak T];\Hr)$.
\end{definition}

\begin{remark}
The word \emph{stable} in \cref{def:stable-reduction} is not used as a synonym for spectral stability of an equilibrium. It refers to stability of the reduction operation itself: controlled changes in the representation of unresolved influence should not produce uncontrolled changes in the resolved evolution.
\end{remark}

\begin{remark}[Level of the present definition]\label{rem:level-one}
\Cref{def:stable-reduction} is a trajectory-based notion. It does not by itself assert preservation of an energy or dissipation structure. The passive class in \cref{cor:passive-energy} gives a concrete stronger requirement.
\end{remark}

\begin{theorem}[Finite-horizon bound]\label{thm:finite-horizon-bound}
Let \cref{eq:semigroup-bound} hold and let
\begin{align*}
  K\in \mathscr K_{\mathfrak T}(\Hr),
  \quad
  F\in L^1(0,\mathfrak T;\Hr).
\end{align*}
Then, \cref{eq:volterra-reduced} has a unique mild solution $z\in \mathcal{C}([0,\mathfrak T];\Hr)$. Furthermore, with
\begin{align*}
  \sigma_{\mathfrak T}:=M\e^{\max\{\omega,0\}\mathfrak T},
  \quad
  \kappa_{\mathfrak T}:=\norm{K}_{\mathscr K_{\mathfrak T}(\Hr)},
\end{align*}
one has
\begin{align}\label{eq:finite-horizon-bound}
  \norm{z}_{\mathcal{C}([0,\mathfrak T];\Hr)}
  \le
  \sigma_{\mathfrak T}
  \left(
    \norm{z_0}_{\Hr}
    +\norm{F}_{L^1(0,\mathfrak T;\Hr)}
  \right)
  \exp\!\left(\sigma_{\mathfrak T}\kappa_{\mathfrak T} \mathfrak T\right).
\end{align}
\end{theorem}

\begin{proof}
The mild form is
\begin{align}\label{eq:mild-volterra}
  z(t)= S_A(t)z_0
  +\int_0^t S_A(t-s)F(s)\,\diff s +\int_0^t S_A(t-s)
     \int_0^s K(s-r)z(r)\,\diff r\,\diff s.
\end{align}
For completeness, we also show the successive-approximation argument. We define
\begin{align*}
  b(t):=S_A(t)z_0+\int_0^tS_A(t-s)F(s)\,\diff s
\end{align*}
and, on $\mathcal{C}([0,\mathfrak T];\Hr)$,
\begin{align*}
  (\mathcal Vv)(t)
  :=\int_0^tS_A(t-s)(K*v)(s)\,\diff s.
\end{align*}
The standard continuity property of semigroup convolutions with $L^1$-data, together with \cref{lem:kernel-convolution}, shows that $b\in \mathcal{C}([0,\mathfrak T];\Hr)$ and that $\mathcal V$ maps $\mathcal{C}([0,\mathfrak T];\Hr)$ into itself. Furthermore, for $0\le t\le \mathfrak T$,
\begin{align*}
  \sup_{0\le s\le t}\norm{(\mathcal Vv)(s)}_{\Hr}
  \le
  \sigma_{\mathfrak T}\kappa_{\mathfrak T}
  \int_0^t
  \sup_{0\le \tau\le r}\norm{v(\tau)}_{\Hr}\,\diff r.
\end{align*}
Iteration therefore gives
\begin{align*}
  \norm{\mathcal V^m b}_{\mathcal{C}([0,t];\Hr)}
  \le
  \frac{(\sigma_{\mathfrak T}\kappa_{\mathfrak T}t)^m}{m!}
  \norm{b}_{\mathcal{C}([0,\mathfrak T];\Hr)}.
\end{align*}
Therefore, the Neumann--Volterra series $\sum_{m=0}^{\infty}\mathcal V^m b$ converges uniformly and yields a mild solution of \cref{eq:mild-volterra}. This is the usual successive-approximation construction for Volterra equations; compare \cite{GripenbergLondenStaffans1990,Pruss1993}.

We set $u(t):=\norm{z(t)}_{\Hr}$. From \cref{eq:semigroup-bound,eq:mild-volterra}, for $0\le t\le \mathfrak T$,
\begin{align*}
\begin{split}
  u(t)
  \le{}&
  \sigma_{\mathfrak T}
  \left(
    \norm{z_0}_{\Hr}
    +\norm{F}_{L^1(0,\mathfrak T;\Hr)}
  \right)
  +\sigma_{\mathfrak T}
  \int_0^t\int_0^s
  \norm{K(s-r)}_{\Lop(\Hr)}u(r)
  \,\diff r\,\diff s.
\end{split}
\end{align*}
Changing the order of integration gives
\begin{align*}
  \int_0^t\int_0^s
  \norm{K(s-r)}_{\Lop(\Hr)}u(r)
  \,\diff r\,\diff s
  \le
  \kappa_{\mathfrak T}\int_0^t u(r)\,\diff r.
\end{align*}
Therefore, Gronwall's inequality yields \cref{eq:finite-horizon-bound}. Applying the same estimate to the difference of two solutions proves uniqueness.
\end{proof}

\begin{theorem}[Kernel stability]\label{thm:kernel-stability}
For $i=1,2$, let $z_i$ solve
\begin{align*}
  \frac{\mathrm{d}}{\mathrm{d}t}z_i(t)
  =Az_i(t)+(K_i*z_i)(t)+F_i(t),
  \quad z_i(0)=z_{0,i},
\end{align*}
under the assumptions of \cref{thm:finite-horizon-bound}. We set
\begin{align*}
  \kappa_*:=\max_{i=1,2}
  \norm{K_i}_{\mathscr K_{\mathfrak T}(\Hr)}.
\end{align*}
Then,
\begin{align}\label{eq:kernel-stability}
\begin{split}
  \norm{z_1-z_2}_{\mathcal{C}([0,\mathfrak T];\Hr)}
  &\le
  \sigma_{\mathfrak T}
  \exp\!\left(\sigma_{\mathfrak T}\kappa_*\mathfrak T\right)
  \Bigl[
    \norm{z_{0,1}-z_{0,2}}_{\Hr}
    +\norm{F_1-F_2}_{L^1(0,\mathfrak T;\Hr)} \\
    &\quad
    +\mathfrak T\norm{K_1-K_2}_{\mathscr K_{\mathfrak T}(\Hr)}
      \norm{z_2}_{\mathcal{C}([0,\mathfrak T];\Hr)}
  \Bigr].
\end{split}
\end{align}
In particular, if
\begin{align*}
  z_{0,n}\to z_0,
  \quad
  F_n\to F
  \quad\text{in }L^1(0,\mathfrak T;\Hr),
  \quad
  K_n\to K
  \quad\text{in }\mathscr K_{\mathfrak T}(\Hr),
\end{align*}
then the corresponding solutions satisfy
\begin{align*}
  z_n\to z
  \quad\text{in }\mathcal{C}([0,\mathfrak T];\Hr).
\end{align*}
\end{theorem}

\begin{proof}
We set $w:=z_1-z_2$. Subtracting the two mild equations gives a term containing $K_1*w$ and the consistency term
\begin{align*}
  ((K_1-K_2)*z_2)(t).
\end{align*}
Arguing as in the proof of \cref{thm:finite-horizon-bound},
\begin{align*}
\begin{split}
  \norm{w(t)}_{\Hr}
  \le{}&
  \sigma_{\mathfrak T}
  \left(
    \norm{z_{0,1}-z_{0,2}}_{\Hr}
    +\norm{F_1-F_2}_{L^1(0,\mathfrak T;\Hr)}
  \right)\\
  &+\sigma_{\mathfrak T}\kappa_*
  \int_0^t\norm{w(r)}_{\Hr}\,\diff r
  +\sigma_{\mathfrak T} \mathfrak T
  \norm{K_1-K_2}_{\mathscr K_{\mathfrak T}(\Hr)}
  \norm{z_2}_{\mathcal{C}([0,\mathfrak T];\Hr)}.
\end{split}
\end{align*}
Gronwall's inequality proves \cref{eq:kernel-stability}. The convergence statement follows from \cref{thm:finite-horizon-bound}, which gives a uniform bound for $z_n$ once the data converge.
\end{proof}

\begin{corollary}[A sufficient criterion for stable reduction]\label{cor:stable-reduction}
Suppose that a reduction family $(K_n,F_n,z_{0,n})$ satisfies
\begin{align*}
  \sup_n\norm{K_n}_{\mathscr K_{\mathfrak T}(\Hr)}<\infty,
  \quad
  \sup_n\norm{F_n}_{L^1(0,\mathfrak T;\Hr)}<\infty,
  \quad
  \sup_n\norm{z_{0,n}}_{\Hr}<\infty,
\end{align*}
and converges to $(K,F,z_0)$ in the data topology of \cref{def:stable-reduction}. Then \cref{eq:reduction-family} is a stable reduction on $[0,\mathfrak T]$ in the sense of \cref{def:stable-reduction}.
\end{corollary}

\begin{remark}[Why the topology matters]\label{rem:topology}
The topology in \cref{def:stable-reduction} is part of the model. Integrated operator-norm convergence is sufficient for the present finite-horizon theory, but different limiting mechanisms may require different topologies. No such extension is used in the results below.
\end{remark}

\section{Stable reduction at the level of hidden realisations}\label{sec:realisation}
The preceding section starts from a memory kernel and asks whether an approximation of that kernel gives a controlled approximation of the resolved trajectory. For unresolved dynamics, however, the more primitive objects are the hidden propagator, the two coupling operators, the hidden initial state, and the hidden forcing. We now connect perturbations of these objects directly to the reduced evolution.

Throughout this section the unresolved space $\Hu$ is fixed. This restriction is made only to keep the first realisation-level theorem transparent. Approximation between different hidden spaces requires identification, lifting, or projection operators and is a separate layer of analysis.

\subsection{Perturbing the hidden propagator and couplings}
For $i\in\{1,2\}$, let $S_i(t)$ be strongly continuous semigroups on $\Hu$, and let
\begin{align*}
  B_i\in\Lop(\Hu,\Hr),
  \quad
  C_i\in\Lop(\Hr,\Hu).
\end{align*}
Let $y_{0,i}\in\Hu$ and $g_i\in L^1(0,\mathfrak T;\Hu)$. The influence generated by the $i$-th hidden realisation is
\begin{align}
  K_i(t)&:=B_iS_i(t)C_i,\label{eq:realisation-kernel}\\
  F_{{\rm u},i}(t)
  &:=B_iS_i(t)y_{0,i}
   +B_i\int_0^tS_i(t-s)g_i(s)\,\diff s.\label{eq:realisation-forcing}
\end{align}
We define the propagator discrepancy
\begin{align}\label{eq:semigroup-discrepancy}
  \delta_S(\mathfrak T)
  :=\int_0^{\mathfrak T}
     \norm{S_1(t)-S_2(t)}_{\Lop(\Hu)}\,\diff t.
\end{align}

\begin{proposition}[Propagation of hidden-realisation errors]\label{prop:hidden-influence-bound}
Assume that
\begin{align}\label{eq:hidden-uniform-bound}
  \norm{S_i(t)}_{\Lop(\Hu)}\le \eta_{\mathfrak T}
  \quad
  (0\le t\le \mathfrak T,\ i=1,2)
\end{align}
for some $\eta_{\mathfrak T}\ge1$. Then,
\begin{align}\label{eq:hidden-kernel-bound}
\begin{split}
  \norm{K_1-K_2}_{\mathscr K_{\mathfrak T}(\Hr)}
  \le{}&
  \mathfrak T\eta_{\mathfrak T}\norm{B_1-B_2}_{\Lop(\Hu,\Hr)}\norm{C_1}_{\Lop(\Hr,\Hu)}\\
  &+\norm{B_2}_{\Lop(\Hu,\Hr)}\,\delta_S(\mathfrak T)\norm{C_1}_{\Lop(\Hr,\Hu)}\\
  &+\mathfrak T\eta_{\mathfrak T}\norm{B_2}_{\Lop(\Hu,\Hr)}\norm{C_1-C_2}_{\Lop(\Hr,\Hu)}.
\end{split}
\end{align}
Furthermore,
\begin{align}\label{eq:hidden-forcing-bound}
\begin{split}
  \norm{F_{{\rm u},1}-F_{{\rm u},2}}_{L^1(0,\mathfrak T;\Hr)}
  \le{}&
  \Bigl[
    \mathfrak T\eta_{\mathfrak T}\norm{B_1-B_2}_{\Lop(\Hu,\Hr)}
    +\norm{B_2}_{\Lop(\Hu,\Hr)}\delta_S(\mathfrak T)
  \Bigr]\\
  &\quad\times
  \Bigl(
    \norm{y_{0,1}}_{\Hu}
    +\norm{g_1}_{L^1(0,\mathfrak T;\Hu)}
  \Bigr)\\
  &+\mathfrak T\eta_{\mathfrak T}\norm{B_2}_{\Lop(\Hu,\Hr)}
  \Bigl(
    \norm{y_{0,1}-y_{0,2}}_{\Hu}
    +\norm{g_1-g_2}_{L^1(0,\mathfrak T;\Hu)}
  \Bigr).
\end{split}
\end{align}
\end{proposition}

\begin{proof}
For the kernels, use the exact decomposition
\begin{align*}
\begin{split}
  B_1S_1C_1-B_2S_2C_2
  ={}&(B_1-B_2)S_1C_1\\
     &+B_2(S_1-S_2)C_1
      +B_2S_2(C_1-C_2).
\end{split}
\end{align*}
Integration over $(0,\mathfrak T)$, together with \cref{eq:hidden-uniform-bound}, gives \cref{eq:hidden-kernel-bound}.

For the first term in \cref{eq:realisation-forcing}, write
\begin{align*}
\begin{split}
 B_1S_1y_{0,1}-B_2S_2y_{0,2}
 ={}&(B_1-B_2)S_1y_{0,1} +B_2(S_1-S_2)y_{0,1}
 +B_2S_2(y_{0,1}-y_{0,2}).
\end{split}
\end{align*}
Its $L^1(0,\mathfrak T;\Hr)$-norm is bounded by the corresponding three terms in \cref{eq:hidden-forcing-bound}. For the convolution part, Young's inequality on $(0,\mathfrak T)$ gives
\begin{align*}
  \norm{S_i*g}_{L^1(0,\mathfrak T;\Hu)}
  \le \mathfrak T\eta_{\mathfrak T}\norm{g}_{L^1(0,\mathfrak T;\Hu)},
\end{align*}
and
\begin{align*}
  \norm{(S_1-S_2)*g}_{L^1(0,\mathfrak T;\Hu)}
  \le \delta_S(\mathfrak T)\norm{g}_{L^1(0,\mathfrak T;\Hu)}.
\end{align*}
Applying the same three-term decomposition proves \cref{eq:hidden-forcing-bound}.
\end{proof}

\begin{definition}[Resolved influence and finite-horizon equivalence]\label{def:resolved-influence}
A hidden realisation on $[0,\mathfrak T]$ is the collection
\begin{align*}
  \mathfrak R=(S,B,C,y_0,g)
\end{align*}
which induces $(K,F_{\rm u})$ through \cref{eq:realisation-kernel,eq:realisation-forcing}. For two such realisations define
\begin{align}\label{eq:influence-pseudometric}
  \mathfrak d_{\mathfrak T}(\mathfrak R_1,\mathfrak R_2)
  :=
  \norm{K_1-K_2}_{\mathscr K_{\mathfrak T}(\Hr)}
  +
  \norm{F_{{\rm u},1}-F_{{\rm u},2}}_{L^1(0,\mathfrak T;\Hr)}.
\end{align}
We call $\mathfrak R_1$ and $\mathfrak R_2$ \emph{resolved-equivalent on $[0,\mathfrak T]$} if
\begin{align*}
  \mathfrak d_{\mathfrak T}(\mathfrak R_1,\mathfrak R_2)=0.
\end{align*}
Thus, $\mathfrak d_{\mathfrak T}$ is a pseudometric on realisations and becomes a metric after quotienting by resolved equivalence.
\end{definition}

\begin{remark}[A reduction acts on influence classes]\label{rem:influence-classes}
Different hidden systems may have the same action on the resolved variables. In that case, distinguishing them is unnecessary for the resolved evolution on the chosen time interval. \Cref{def:resolved-influence} makes this explicit: the object seen by the resolved equation is not the hidden realisation itself but its induced pair $(K,F_{\rm u})$. This is consistent with the classical realisation viewpoint in linear systems, where distinct state-space representations can describe the same input--output map; see \cite{Staffans2005}. Here the unresolved initial state and hidden forcing are included because they also contribute to the resolved trajectory.
\end{remark}

\begin{theorem}[Realisation-level stable reduction]\label{thm:realisation-stability}
Let $A$ satisfy \cref{eq:semigroup-bound} and let $f\in L^1(0,\mathfrak T;\Hr)$. Let $x_i\in \mathcal{C}([0,\mathfrak T];\Hr)$ be the resolved trajectories associated with the two hidden realisations, with common resolved initial value $x_0$ and forcing $f$:
\begin{align*}
  \frac{\mathrm{d}}{\mathrm{d}t}x_i(t)
  =Ax_i(t)+(K_i*x_i)(t)+F_{{\rm u},i}(t)+f(t),
  \quad
  x_i(0)=x_0.
\end{align*}
We set
\begin{align*}
  \kappa_*:=\max_{i=1,2}
  \norm{K_i}_{\mathscr K_{\mathfrak T}(\Hr)}.
\end{align*}
Then,
\begin{align}\label{eq:realisation-stability}
\begin{split}
 \norm{x_1-x_2}_{\mathcal{C}([0,\mathfrak T];\Hr)}
 \le{}&
 \sigma_{\mathfrak T}\exp\!\left(\sigma_{\mathfrak T}\kappa_*\mathfrak T\right)
 \Bigl[
  \norm{F_{{\rm u},1}-F_{{\rm u},2}}_{L^1(0,\mathfrak T;\Hr)}\\
 &\quad
  +\mathfrak T\norm{K_1-K_2}_{\mathscr K_{\mathfrak T}(\Hr)}
   \norm{x_2}_{\mathcal{C}([0,\mathfrak T];\Hr)}
 \Bigr].
\end{split}
\end{align}
Consequently, suppose that a sequence of hidden realisations satisfies, for any fixed $\mathfrak T>0$,
\begin{align*}
\begin{gathered}
  \eta_{\mathfrak T}^*:=
  \max\left\{
    \sup_n\sup_{0\le t\le \mathfrak T}\norm{S_n(t)}_{\Lop(\Hu)},
    \sup_{0\le t\le \mathfrak T}\norm{S(t)}_{\Lop(\Hu)}
  \right\}<\infty,\\
  \int_0^{\mathfrak T}\norm{S_n(t)-S(t)}_{\Lop(\Hu)}\,\diff t\longrightarrow0,\\
  \norm{B_n-B}_{\Lop(\Hu,\Hr)}\to0,
  \quad \norm{C_n-C}_{\Lop(\Hr,\Hu)}\to0,\\
  \norm{y_{0,n}-y_0}_{\Hu}\to0,
  \quad \norm{g_n-g}_{L^1(0,\mathfrak T;\Hu)}\to0.
\end{gathered}
\end{align*}
Then, the corresponding resolved trajectories converge in $\mathcal{C}([0,\mathfrak T];\Hr)$.
\end{theorem}

\begin{proof}
Estimate \cref{eq:realisation-stability} is \cref{eq:kernel-stability} with identical resolved initial data and with $F_i=F_{{\rm u},i}+f$. The convergence statement follows from \cref{prop:hidden-influence-bound}: its right-hand sides tend to zero, while \cref{thm:finite-horizon-bound} gives a uniform bound for $x_n$ because the induced kernels and effective forcings are uniformly bounded on every fixed interval.
\end{proof}

\begin{corollary}[Lipschitz continuity on bounded influence classes]\label{cor:influence-lipschitz}
Fix $\mathfrak T>0$, the resolved generator $A$, the resolved initial value $x_0$, and $f\in L^1(0,\mathfrak T;\Hr)$. Let $\mathfrak R_1,\mathfrak R_2$ be two hidden realisations such that
\begin{align*}
  \norm{K_i}_{\mathscr K_{\mathfrak T}(\Hr)}\le\kappa,
  \quad
  \norm{F_{{\rm u},i}+f}_{L^1(0,\mathfrak T;\Hr)}\le m,
  \quad i=1,2.
\end{align*}
We define
\begin{align*}
  R_{\mathfrak T}
  :=\sigma_{\mathfrak T}\bigl(\norm{x_0}_{\Hr}+m\bigr)
       \exp\!\left(\sigma_{\mathfrak T}\kappa \mathfrak T\right).
\end{align*}
Then, the associated resolved trajectories satisfy
\begin{align}\label{eq:influence-lipschitz}
  \norm{x_1-x_2}_{\mathcal{C}([0,\mathfrak T];\Hr)}
  \le
  \sigma_{\mathfrak T}\exp\!\left(\sigma_{\mathfrak T}\kappa \mathfrak T\right)
  \max\{1,\mathfrak T R_{\mathfrak T}\}
  \mathfrak d_{\mathfrak T}(\mathfrak R_1,\mathfrak R_2).
\end{align}
Consequently, the resolved solution map descends to the quotient by resolved equivalence and is Lipschitz continuous on subsets with uniformly bounded induced kernel and forcing.
\end{corollary}

\begin{proof}
By \cref{thm:finite-horizon-bound}, $\norm{x_i}_{\mathcal{C}([0,\mathfrak T];\Hr)}\le R_{\mathfrak T}$. Insert this bound into \cref{eq:realisation-stability} and use
\begin{align*}
  a+\mathfrak T R_{\mathfrak T}b\le\max\{1,\mathfrak T R_{\mathfrak T}\}(a+b),
  \quad a,b\ge0.
\end{align*}
The last statement follows because $\mathfrak d_{\mathfrak T}=0$ implies equality of the induced kernel and effective forcing, hence equality of the resolved trajectories.
\end{proof}

\begin{corollary}[Approximation of the resolved component of the full system]\label{cor:full-to-reduced}
Let $(x,y)$ solve the full block system \cref{eq:block-system,eq:block-initial} on $[0,\mathfrak T]$, assume $g\in L^1(0,\mathfrak T;\Hu)$, and let $\mathfrak R=(S_D,B,C,y_0,g)$ be its exact hidden realisation. Let $x_n$ be reduced trajectories generated by hidden realisations $\mathfrak R_n$, with the same resolved operator $A$, initial value $x_0$, and resolved forcing $f$. If
\begin{align*}
  \mathfrak d_{\mathfrak T}(\mathfrak R_n,\mathfrak R)\longrightarrow0
\end{align*}
and the induced kernels $K_n$ are uniformly bounded in $\mathscr K_{\mathfrak T}(\Hr)$, then
\begin{align*}
  x_n\longrightarrow x
  \quad\text{in }\mathcal{C}([0,\mathfrak T];\Hr).
\end{align*}
In particular, no separate closure error occurs between the full resolved component and the exact memory equation; the error measured here is entirely due to approximation of the unresolved influence.
\end{corollary}

\begin{proof}
By \cref{prop:exact-elimination}, the full resolved component $x$ is exactly the solution of the Volterra equation induced by $\mathfrak R$. The assertion is therefore a direct consequence of \cref{thm:kernel-stability} and \cref{eq:influence-pseudometric}.
\end{proof}

\begin{remark}[What is gained by returning to the realisation]\label{rem:return-realisation}
Kernel stability alone says that a good approximation of the effective memory gives a good resolved trajectory. \Cref{thm:realisation-stability} identifies one level earlier where that approximation can be controlled. Errors in the hidden propagator, coupling, hidden initial condition, and hidden forcing remain distinct until the final stability estimate. This is useful when the unresolved model itself is approximated before any explicit memory kernel is formed.
\end{remark}

\begin{corollary}[Bounded hidden generators]\label{cor:bounded-generators}
Let $D_n,D\in\Lop(\Hu)$, and suppose
\begin{align*}
  \norm{D_n-D}_{\Lop(\Hu)}\longrightarrow0.
\end{align*}
If
\begin{align*}
  d_*:=\max\!\left\{\sup_n\norm{D_n}_{\Lop(\Hu)},\norm{D}_{\Lop(\Hu)}\right\}<\infty,
\end{align*}
then, with $S_n(t)=\e^{tD_n}$ and $S(t)=\e^{tD}$,
\begin{align}\label{eq:bounded-generator-semigroup}
  \norm{S_n(t)-S(t)}_{\Lop(\Hu)}
  \le t\e^{d_*t}\norm{D_n-D}_{\Lop(\Hu)}
  \quad (t\ge0).
\end{align}
In particular,
\begin{align*}
  \int_0^{\mathfrak T}\norm{S_n(t)-S(t)}_{\Lop(\Hu)}\,\diff t
  \le \frac{\mathfrak T^2}{2}\e^{d_*\mathfrak T}\norm{D_n-D}_{\Lop(\Hu)},
\end{align*}
and the convergence assumptions on the hidden propagator in \cref{thm:realisation-stability} follow from operator-norm convergence of the bounded hidden generators.
\end{corollary}

\begin{proof}
Duhamel's identity gives
\begin{align*}
  \e^{tD_n}-\e^{tD}
  =\int_0^t
    \e^{(t-s)D_n}(D_n-D)\e^{sD}\,\diff s.
\end{align*}
Using $\norm{\e^{tD_n}}_{\Lop(\Hu)},\norm{\e^{tD}}_{\Lop(\Hu)}\le\e^{d_*t}$ proves \cref{eq:bounded-generator-semigroup}. Integration and the estimate $t\e^{d_*t}\le t\e^{d_*\mathfrak T}$ yield the stated $L^1$-bound. This is the elementary bounded-generator form of the standard semigroup perturbation principle; compare \cite{Pazy1983}.
\end{proof}

\begin{remark}[Propagator topology]\label{rem:strong-vs-norm}
Strong convergence $S_n(t)u\to S(t)u$ for each fixed hidden state is weaker than operator-norm control. This matters for Galerkin approximations of infinite-dimensional generators. The theorem above is only a sufficient condition. It is not a general Trotter--Kato theorem. Compact couplings or a trajectory-dependent topology may weaken the requirement. \Cref{prop:finite-passive-truncation} gives one example.
\end{remark}

\subsection{Structure-preserving stable reduction in the passive class}
Trajectory convergence does not by itself preserve a storage or dissipation law. In model reduction this distinction is familiar: when passivity, Hamiltonian structure, or another physical constraint matters, preservation of that structure must be built into the reduction rather than inferred from a small trajectory error. Structure-preserving reduction of passive state-space systems has been developed in several settings; for port-Hamiltonian systems, see, for example, \cite{GugercinPolyugaBeattieVanderSchaft2012}. The present setting is different in that the hidden state is eliminated into memory, but the same separation between approximation quality and structural compatibility is essential. Here the adjective \emph{passive} refers to the hidden feedback realisation and its memory storage--dissipation identity. Unless additional dissipativity assumptions are imposed on the resolved generator $A$, we do not claim passivity of the entire resolved evolution.

\begin{definition}[Structure-preserving stable reduction]\label{def:structure-preserving}
Let $L\in\Lop(\Hu)$ be self-adjoint and non-negative and let $C\in\Lop(\Hr,\Hu)$. Consider a family $(L_n,C_n)$, with $L_n\in\Lop(\Hu)$ and $C_n\in\Lop(\Hr,\Hu)$, and define the corresponding kernels
\begin{align*}
  K_n(t):=C_n^*\e^{-tL_n}C_n.
\end{align*}
For a fixed resolved generator $A$ satisfying \cref{eq:semigroup-bound}, forcing $F\in L^1(0,\mathfrak T;\Hr)$, and initial state $x_0\in\Hr$, let $x_n$ and $x$ denote the corresponding passive reduced trajectories
\begin{align*}
  \frac{\mathrm{d}}{\mathrm{d}t}x_n(t)
  &=Ax_n(t)-(K_n*x_n)(t)+F(t),\\
  \frac{\mathrm{d}}{\mathrm{d}t}x(t)
  &=Ax(t)-(K*x)(t)+F(t),\\
  x_n(0)&=x(0)=x_0.
\end{align*}
and set
\begin{align*}
 q_n(t)&:=\int_0^t\e^{-(t-s)L_n}C_nx_n(s)\,\diff s,\\
 q(t)&:=\int_0^t\e^{-(t-s)L}Cx(s)\,\diff s.
\end{align*}
We call the family a \emph{structure-preserving stable reduction on $[0,\mathfrak T]$} of the passive realisation $(L,C)$ if:
\begin{enumerate}[label=\textup{(P\arabic*)},leftmargin=3.2em]
  \item each $L_n$ is self-adjoint and non-negative, and the feedback is represented by the adjoint pair $C_n^*,C_n$;
  \item $x_n\to x$ in $\mathcal{C}([0,\mathfrak T];\Hr)$ and $q_n\to q$ in $\mathcal{C}([0,\mathfrak T];\Hu)$;
  \item the memory storage--dissipation functional converges:
  \begin{align}\label{eq:structure-functional-convergence}
  \begin{split}
    &\frac12\norm{q_n(\mathfrak T)}_{\Hu}^2
     +\int_0^{\mathfrak T}\norm{L_n^{1/2}q_n(t)}_{\Hu}^2\,\diff t
    \longrightarrow
     \frac12\norm{q(\mathfrak T)}_{\Hu}^2
     +\int_0^{\mathfrak T}\norm{L^{1/2}q(t)}_{\Hu}^2\,\diff t.
  \end{split}
  \end{align}
\end{enumerate}
Thus, level-I trajectory stability is supplemented by convergence of the hidden storage and dissipation that survive in the memory representation.
\end{definition}

\begin{theorem}[Norm-convergent passive realisations preserve structure]\label{thm:passive-stable-reduction}
Assume
\begin{align*}
  L_n=L_n^*\ge0,
  \quad
  L=L^*\ge0,
  \quad
  \norm{L_n-L}_{\Lop(\Hu)}\to0,
\end{align*}
and
\begin{align*}
  \norm{C_n-C}_{\Lop(\Hr,\Hu)}\to0.
\end{align*}
Let $A$ satisfy \cref{eq:semigroup-bound}, fix $F\in L^1(0,\mathfrak T;\Hr)$ and $x_0\in\Hr$, and let $x_n,x$ solve
\begin{align*}
  \frac{\mathrm{d}}{\mathrm{d}t}x_n(t)
  &=Ax_n(t)-(K_n*x_n)(t)+F(t),\\
  \frac{\mathrm{d}}{\mathrm{d}t}x(t)
  &=Ax(t)-(K*x)(t)+F(t),\\
  x_n(0)&=x(0)=x_0.
\end{align*}
where
\begin{align*}
  K_n(t)=C_n^*\e^{-tL_n}C_n,
  \quad
  K(t)=C^*\e^{-tL}C.
\end{align*}
Then, for every fixed $\mathfrak T>0$, $(L_n,C_n)$ is a structure-preserving stable reduction in the sense of \cref{def:structure-preserving}. In particular,
\begin{align*}
  x_n\to x\quad\hbox{in }\mathcal{C}([0,\mathfrak T];\Hr),
  \quad
  q_n\to q\quad\hbox{in }\mathcal{C}([0,\mathfrak T];\Hu),
\end{align*}
and \cref{eq:structure-functional-convergence} holds.
\end{theorem}

\begin{proof}
Because $L_n,L\ge0$, the semigroups $\e^{-tL_n}$ and $\e^{-tL}$ are contractions. Duhamel's identity therefore gives
\begin{align}\label{eq:passive-semigroup-difference}
  \norm{\e^{-tL_n}-\e^{-tL}}_{\Lop(\Hu)}
  \le t\norm{L_n-L}_{\Lop(\Hu)}.
\end{align}
Consequently,
\begin{align}\label{eq:passive-kernel-convergence}
\begin{split}
 \norm{K_n-K}_{\mathscr K_{\mathfrak T}(\Hr)}
 \le{}&\mathfrak T\norm{C_n-C}_{\Lop(\Hr,\Hu)}
       \bigl(\norm{C_n}_{\Lop(\Hr,\Hu)}+\norm{C}_{\Lop(\Hr,\Hu)}\bigr)\\
 &+\frac{\mathfrak T^2}{2}\norm{C}_{\Lop(\Hr,\Hu)}\norm{C_n}_{\Lop(\Hr,\Hu)}
       \norm{L_n-L}_{\Lop(\Hu)},
\end{split}
\end{align}
which tends to zero. Applying \cref{thm:kernel-stability} to the Volterra kernels $-K_n$ and $-K$ gives $x_n\to x$ in $\mathcal{C}([0,\mathfrak T];\Hr)$.

For the memory states, we write
\begin{align*}
\begin{split}
 q_n(t)-q(t)
 ={}&\int_0^t\e^{-(t-s)L_n}C_n\{x_n(s)-x(s)\}\,\diff s\\
 &+\int_0^t\e^{-(t-s)L_n}(C_n-C)x(s)\,\diff s\\
 &+\int_0^t
   \bigl(\e^{-(t-s)L_n}-\e^{-(t-s)L}\bigr)Cx(s)\,\diff s.
\end{split}
\end{align*}
Using contraction and \cref{eq:passive-semigroup-difference},
\begin{align}\label{eq:q-convergence-bound}
\begin{split}
 \norm{q_n-q}_{\mathcal{C}([0,\mathfrak T];\Hu)}
 \le{}&\mathfrak T\norm{C_n}_{\Lop(\Hr,\Hu)}\norm{x_n-x}_{\mathcal{C}([0,\mathfrak T];\Hr)}\\
 &+\mathfrak T\norm{C_n-C}_{\Lop(\Hr,\Hu)}\norm{x}_{\mathcal{C}([0,\mathfrak T];\Hr)}\\
 &+\frac{\mathfrak T^2}{2}\norm{L_n-L}_{\Lop(\Hu)}\norm{C}_{\Lop(\Hr,\Hu)}
       \norm{x}_{\mathcal{C}([0,\mathfrak T];\Hr)},
\end{split}
\end{align}
so $q_n\to q$ uniformly.

It remains to prove convergence of the storage--dissipation functional without assuming any separate convergence theorem for the square roots $L_n^{1/2}$. By \cref{thm:passive-memory},
\begin{align}\label{eq:memory-quadratic-functional}
\begin{split}
 Q_n:={}&\int_0^{\mathfrak T}\inner{(K_n*x_n)(t)}{x_n(t)}\,\diff t\\
 ={}&\frac12\norm{q_n(\mathfrak T)}_{\Hu}^2
     +\int_0^{\mathfrak T}\norm{L_n^{1/2}q_n(t)}_{\Hu}^2\,\diff t,
\end{split}
\end{align}
and the analogous identity holds for $Q$. To make the last limit explicit, decompose
\begin{align*}
\begin{split}
  Q_n-Q
  ={}&\int_0^{\mathfrak T}
       \inner{((K_n-K)*x_n)(t)}{x_n(t)}\,\diff t\\
     &+\int_0^{\mathfrak T}
       \inner{(K*(x_n-x))(t)}{x_n(t)}\,\diff t\\
     &+\int_0^{\mathfrak T}
       \inner{(K*x)(t)}{x_n(t)-x(t)}\,\diff t.
\end{split}
\end{align*}
By \cref{eq:kernel-convolution-bound},
\begin{align*}
\begin{split}
  |Q_n-Q|
  \le{}&\mathfrak T\norm{K_n-K}_{\mathscr K_{\mathfrak T}(\Hr)}
          \norm{x_n}_{\mathcal{C}([0,\mathfrak T];\Hr)}^2\\
  &+\mathfrak T\norm{K}_{\mathscr K_{\mathfrak T}(\Hr)}
    \bigl(
      \norm{x_n}_{\mathcal{C}([0,\mathfrak T];\Hr)}
      +\norm{x}_{\mathcal{C}([0,\mathfrak T];\Hr)}
    \bigr)
    \norm{x_n-x}_{\mathcal{C}([0,\mathfrak T];\Hr)}.
\end{split}
\end{align*}
The right-hand side tends to zero because $x_n\to x$ uniformly and $K_n\to K$ in $\mathscr K_{\mathfrak T}(\Hr)$. Hence $Q_n\to Q$. Together with $q_n(\mathfrak T)\to q(\mathfrak T)$, this also yields convergence of the dissipation term separately, and in particular proves \cref{eq:structure-functional-convergence}.
\end{proof}

\begin{remark}[Uniform structural identities, not only limiting positivity]\label{rem:uniform-structure}
Every approximating kernel $K_n$ in \cref{thm:passive-stable-reduction} is itself of positive type and satisfies the exact storage identity before the limit is taken. The theorem therefore does more than show that a sequence of arbitrary kernels happens to converge to a positive one. The approximations remain inside the passive class throughout the reduction. This is the sense in which the reduction is structure-preserving here.
\end{remark}

\subsection{Finite hidden-state truncation}

The preceding theorem gives a direct route to finite internal-variable models when the passive hidden dynamics admits compatible finite-dimensional spectral subspaces.

\begin{proposition}[Finite-mode passive truncation]\label{prop:finite-passive-truncation}
Let $L=L^*\ge0$ be bounded on $\Hu$, and suppose that there is an increasing sequence of finite-rank orthogonal projections $P_N$ such that
\begin{align*}
  P_NL=LP_N,
  \quad
  P_N\to\Id\quad\hbox{strongly in }\Hu.
\end{align*}
Let $C\in\Lop(\Hr,\Hu)$ be compact and define
\begin{align*}
  C_N:=P_NC,
  \quad
  K_N(t):=C_N^*\e^{-tL}C_N.
\end{align*}
Then,
\begin{align*}
  \delta_N:=\norm{(\Id-P_N)C}_{\Lop(\Hr,\Hu)}\longrightarrow0
\end{align*}
and
\begin{align}\label{eq:finite-truncation-kernel}
  \norm{K_N-K}_{\mathscr K_{\mathfrak T}(\Hr)}
  \le 2\mathfrak T\norm{C}_{\Lop(\Hr,\Hu)}\delta_N.
\end{align}
Therefore, for any fixed resolved generator $A$ satisfying \cref{eq:semigroup-bound}, any $F\in L^1(0,\mathfrak T;\Hr)$, and every $x_0\in\Hr$, the associated passive reduced trajectories converge on $[0,\mathfrak T]$, and the family is structure-preserving in the sense of \cref{def:structure-preserving}.

Furthermore, because $P_N\Hu$ is finite-dimensional and invariant under $L$, if
\begin{align*}
  L|_{P_N\Hu}
  =\sum_{j=1}^{m_N}\lambda_{j,N}Q_{j,N}
\end{align*}
is its spectral decomposition, then
\begin{align}\label{eq:positive-prony-truncation}
  K_N(t)
  =\sum_{j=1}^{m_N}
    \e^{-\lambda_{j,N}t}
    C_N^*Q_{j,N}C_N,
\end{align}
where every weight $C_N^*Q_{j,N}C_N$ is self-adjoint and non-negative.
\end{proposition}

\begin{proof}
Strong convergence $P_N\to\Id$ is uniform on compact subsets of $\Hu$. Since $C$ is compact, the image of the unit ball of $\Hr$ has compact closure, and therefore $\delta_N\to0$. Since $P_N$ commutes with $L$, it also commutes with $\e^{-tL}$, and the range of $C_N$ evolves inside the finite-dimensional subspace $P_N\Hu$. Using $\norm{\e^{-tL}}_{\Lop(\Hu)}\le1$,
\begin{align*}
\begin{split}
 K(t)-K_N(t)
 ={}&(C-C_N)^*\e^{-tL}C
     +C_N^*\e^{-tL}(C-C_N),
\end{split}
\end{align*}
which gives \cref{eq:finite-truncation-kernel}. The convergence and structure-preserving conclusions follow from \cref{thm:passive-stable-reduction} with $L_N=L$ and $C_N=P_NC$. Moreover, the corresponding memory state $q_N(t)$ lies in $P_N\Hu$ for every $t$, because $P_N\Hu$ is invariant under $\e^{-tL}$; hence the internal-variable part of the realisation is genuinely finite-dimensional. Finally, the finite-dimensional spectral theorem yields \cref{eq:positive-prony-truncation}; its weights are non-negative because
\begin{align*}
  \inner{C_N^*Q_{j,N}C_Nv}{v}_{\Hr}
  =\norm{Q_{j,N}C_Nv}_{\Hu}^2\ge0.
\end{align*}
\end{proof}

\begin{remark}[Why compactness appears]\label{rem:compact-coupling}
The assumption that $C$ is compact is not a generic requirement of stable reduction. It is used here for one precise purpose: strong convergence of the finite-rank projections $P_N$ then becomes operator-norm convergence of $P_NC$ to $C$. Without such compactness, spectral truncation may converge strongly on individual hidden inputs without converging in the integrated operator norm required by the present $\mathscr K_{\mathfrak T}$-kernel theorem. A weaker topology would require a correspondingly different trajectory-stability argument.
\end{remark}

\section{Elementary examples}\label{sec:examples}

\subsection{A hidden scalar mode generates exponential memory}

We consider
\begin{subequations}\label{eq:scalar-memory-system}
\begin{align}
  \frac{\mathrm{d}}{\mathrm{d}t}x(t)&=-a x(t)+b y(t),\\
  \frac{\mathrm{d}}{\mathrm{d}t}y(t)&=c x(t)-d y(t),
\end{align}
\end{subequations}
where $a,d>0$ and $b,c\in\R$. The unresolved equation gives
\begin{align*}
  y(t)=\e^{-dt}y_0
       +c\int_0^t\e^{-d(t-s)}x(s)\,\diff s.
\end{align*}
Therefore, the exact equation for $x$ is
\begin{align}\label{eq:scalar-memory-reduced}
  \frac{\mathrm{d}}{\mathrm{d}t}x(t)
  =-a x(t)
   +bc\int_0^t\e^{-d(t-s)}x(s)\,\diff s
   +b\e^{-dt}y_0.
\end{align}
The unresolved variable has disappeared from the state vector, but its two effects remain visible: the feedback kernel
\begin{align*}
  K(t)=bc\e^{-dt}
\end{align*}
and the initial-state forcing $b\e^{-dt}y_0$. Deleting $y$ and replacing \cref{eq:scalar-memory-system} by $\frac{\mathrm{d}}{\mathrm{d}t}x(t)=-a x(t)$ therefore changes the dynamics unless these two terms are known to be negligible in a relevant topology.

The example also shows why internal-variable and memory representations are two descriptions of the same mechanism. Starting from \cref{eq:scalar-memory-reduced}, the exponential convolution can be localised again by introducing the internal variable
\begin{align*}
  q(t):=\int_0^t\e^{-d(t-s)}x(s)\,\diff s,
  \quad
  \frac{\mathrm{d}}{\mathrm{d}t}q(t)=x(t)-dq(t).
\end{align*}
Thus elimination creates memory, while a suitable internal variable can remove the explicit memory at the price of enlarging the state space. This duality is familiar in the analysis of equations with memory; compare \cite{Pruss1993}.

\subsection{Small unresolved amplitude can accumulate}

The next example is elementary but important.

\begin{example}[Small state, order-one long-time effect]\label{ex:cumulative}
For $\varepsilon>0$, we consider
\begin{align}\label{eq:cumulative-system}
  \frac{\mathrm{d}}{\mathrm{d}t}x(t)=y(t),
  \quad
  \frac{\mathrm{d}}{\mathrm{d}t}y(t)=0,
  \quad
  x(0)=0,
  \quad
  y(0)=\varepsilon.
\end{align}
Then,
\begin{align*}
  y(t)=\varepsilon,
  \quad
  x(t)=\varepsilon t.
\end{align*}
The unresolved component is uniformly small:
\begin{align*}
  \sup_{t\ge0}|y(t)|=\varepsilon.
\end{align*}
If it is simply deleted, the reduced equation is $\frac{\mathrm{d}}{\mathrm{d}t}x_{\mathrm{tr}}(t)=0$ and therefore $x_{\mathrm{tr}}(t)=0$. At the time $\mathfrak T_\varepsilon=\varepsilon^{-1}$, however,
\begin{align*}
  |x(\mathfrak T_\varepsilon)-x_{\mathrm{tr}}(\mathfrak T_\varepsilon)|=1.
\end{align*}
Thus, small unresolved amplitude does not imply uniformly small influence on time intervals whose length grows as the unresolved scale decreases.
\end{example}

\begin{remark}
\Cref{ex:cumulative} does not say that small unresolved variables can never be neglected. It says that the criterion must involve the coupling, the time horizon, and the stability structure of the evolution, not the amplitude of the unresolved state alone. This distinction becomes more important in multiscale systems, where small or rapidly varying components may produce finite drift, diffusion, memory, or effective dissipation after reduction; see, for example, \cite{GivonKupfermanStuart2004,PavliotisStuart2008}.
\end{remark}

\section{Interpretation and role of the present note}\label{sec:meaning}

The ingredients used above belong to established theories. Exact elimination and the appearance of memory are classical in projection-operator and Mori--Zwanzig approaches \cite{Mori1965,Zwanzig1961,ChorinHaldKupferman2000,ChorinHaldKupferman2002}; Volterra equations provide a classical analytic language for memory \cite{GripenbergLondenStaffans1990,Pruss1993}; and structure-preserving model reduction has a substantial theory of its own \cite{GugercinPolyugaBeattieVanderSchaft2012}. The purpose of the present note is therefore not to relabel any one of these theories.

Its role is to isolate a common reduction question in the chain
\begin{align*}
  \text{hidden dynamics}
  \longrightarrow
  (K,F_{\mathrm u})
  \longrightarrow
  \text{resolved trajectory}
  \longrightarrow
  \text{retained structure}.
\end{align*}
The first arrow is exact elimination. The second asks how accurately the induced unresolved influence must be represented in order to control the resolved trajectory. The third asks whether a structural property that survives elimination is also preserved by the approximating family.

In particular, the object being reduced need not be the hidden state itself. By \cref{def:resolved-influence}, different hidden realisations may be indistinguishable to the resolved equation when they induce the same pair $(K,F_{\mathrm u})$. The present theory concerns approximation of this resolved influence; it does not claim unique reconstruction of the hidden mechanism from resolved data.

The proved hierarchy is deliberately modest. Level I gives finite-horizon trajectory stability under integrated operator-norm perturbations of the unresolved influence. The realisation-level estimate propagates errors in the hidden propagator, couplings, hidden initial state, and hidden forcing to the resolved trajectory. Level II adds convergence of a storage--dissipation functional in a passive subclass, and the finite-mode result gives one explicit structure-preserving approximation family. In this sense the note fixes a minimal principle:
\begin{align*}
  \boxed{
  \text{remove unresolved variables}
  \quad\text{without uncontrolled loss of their dynamical influence}
  }.
\end{align*}

\section{What stable reduction does and does not mean}\label{sec:scope}

The terminology \emph{stable reduction} is useful only if its scope is kept clear.

\paragraph{Reduction is not truncation.}
A truncation removes variables and their couplings. Exact elimination may remove variables from the state vector while retaining their influence through memory and effective forcing; \cref{prop:exact-elimination} is the basic example.

\paragraph{Stable reduction is not necessarily Markovian.}
A Markovian reduced model may be appropriate after an additional approximation or asymptotic argument, but Markovianity is not a consequence of eliminating unresolved variables.

\paragraph{Stable reduction is not determined by state amplitude alone.}
The relevant quantity is the influence transferred to the resolved dynamics. It depends on the coupling, the unresolved propagator, the observation horizon, and the stability of the resolved equation; see \cref{ex:cumulative}.

\paragraph{Trajectory stability and structure preservation are different requirements.}
\Cref{def:stable-reduction} controls resolved trajectories on a fixed interval, whereas \cref{def:structure-preserving} adds a storage--dissipation requirement in the passive class. Neither statement should be inferred from the other without additional hypotheses.

Finally, exact elimination is not the same as approximate closure. If the effective forcing in \cref{prop:exact-elimination} is retained, the exact memory equation has the same resolved trajectory as the full block system. Approximation enters only when that exact influence is altered. The present note therefore does not propose a universal reduction algorithm; it provides a criterion for deciding what a proposed reduction must control.

\paragraph{Acknowledgements.}
None.

\paragraph{Funding.}
The author declares that no funds, grants, or other support were received during the preparation of this manuscript.

\paragraph{Data availability.}
No datasets were generated or analysed during the current study.

\paragraph{Competing interests.}
The author declares no competing interests.



\begin{thebibliography}{99}

\bibitem{BennerGugercinWillcox2015}
Benner, P., Gugercin, S., Willcox, K.: A survey of projection-based model reduction methods for parametric dynamical systems. SIAM Rev. \textbf{57}(4), 483--531 (2015). \url{https://doi.org/10.1137/130932715}

\bibitem{ChorinHaldKupferman2000}
Chorin, A.J., Hald, O.H., Kupferman, R.: Optimal prediction and the Mori--Zwanzig representation of irreversible processes. Proc. Natl. Acad. Sci. USA \textbf{97}(7), 2968--2973 (2000). \url{https://doi.org/10.1073/pnas.97.7.2968}

\bibitem{ChorinHaldKupferman2002}
Chorin, A.J., Hald, O.H., Kupferman, R.: Optimal prediction with memory. Physica D \textbf{166}(3--4), 239--257 (2002). \url{https://doi.org/10.1016/S0167-2789(02)00446-3}

\bibitem{ChorinKastKupferman1999}
Chorin, A.J., Kast, A.P., Kupferman, R.: Unresolved computation and optimal predictions. Commun. Pure Appl. Math. \textbf{52}(10), 1231--1254 (1999)

\bibitem{Dafermos1970}
Dafermos, C.M.: An abstract Volterra equation with applications to linear viscoelasticity. J. Differential Equations \textbf{7}(3), 554--569 (1970). \url{https://doi.org/10.1016/0022-0396(70)90101-4}

\bibitem{GivonHaldKupferman2005}
Givon, D., Hald, O.H., Kupferman, R.: Existence proof for orthogonal dynamics and the Mori--Zwanzig formalism. Israel J. Math. \textbf{145}, 221--241 (2005). \url{https://doi.org/10.1007/BF02786691}

\bibitem{GivonKupfermanStuart2004}
Givon, D., Kupferman, R., Stuart, A.: Extracting macroscopic dynamics: model problems and algorithms. Nonlinearity \textbf{17}(6), R55--R127 (2004). \url{https://doi.org/10.1088/0951-7715/17/6/R01}

\bibitem{GripenbergLondenStaffans1990}
Gripenberg, G., Londen, S.-O., Staffans, O.: \emph{Volterra Integral and Functional Equations}. Encyclopedia of Mathematics and its Applications, vol. 34. Cambridge University Press, Cambridge (1990)

\bibitem{GugercinPolyugaBeattieVanderSchaft2012}
Gugercin, S., Polyuga, R.V., Beattie, C., van der Schaft, A.: Structure-preserving tangential interpolation for model reduction of port-Hamiltonian systems. Automatica \textbf{48}(9), 1963--1974 (2012). \url{https://doi.org/10.1016/j.automatica.2012.05.052}

\bibitem{Ishizaka2026Memory}
Ishizaka, H.: Weighted well-posedness and kernel stability for coercive evolution equations with measure-valued delays. arXiv:2602.19099 (2026). \url{https://doi.org/10.48550/arXiv.2602.19099}

\bibitem{Mori1965}
Mori, H.: Transport, collective motion, and Brownian motion. Prog. Theor. Phys. \textbf{33}(3), 423--455 (1965). \url{https://doi.org/10.1143/PTP.33.423}

\bibitem{PavliotisStuart2008}
Pavliotis, G.A., Stuart, A.M.: \emph{Multiscale Methods: Averaging and Homogenization}. Texts in Applied Mathematics, vol. 53. Springer, New York (2008)

\bibitem{Pazy1983}
Pazy, A.: \emph{Semigroups of Linear Operators and Applications to Partial Differential Equations}. Applied Mathematical Sciences, vol. 44. Springer, New York (1983). \url{https://doi.org/10.1007/978-1-4612-5561-1}

\bibitem{Pruss1993}
Pr\"uss, J.: \emph{Evolutionary Integral Equations and Applications}. Monographs in Mathematics, vol. 87. Birkh\"auser, Basel (1993). \url{https://doi.org/10.1007/978-3-0348-8570-6}

\bibitem{SchillingSongVondracek2012}
Schilling, R.L., Song, R., Vondra\v{c}ek, Z.: \emph{Bernstein Functions: Theory and Applications}, 2nd edn. De Gruyter, Berlin/Boston (2012). \url{https://doi.org/10.1515/9783110269338}

\bibitem{Staffans2005}
Staffans, O.: \emph{Well-Posed Linear Systems}. Encyclopedia of Mathematics and its Applications, vol. 103. Cambridge University Press, Cambridge (2005). \url{https://doi.org/10.1017/CBO9780511543197}

\bibitem{Zwanzig1961}
Zwanzig, R.: Memory effects in irreversible thermodynamics. Phys. Rev. \textbf{124}, 983--992 (1961). \url{https://doi.org/10.1103/PhysRev.124.983}

\end{thebibliography}
\end{document}